\documentclass[a4paper]{llncs}
\usepackage{amsfonts,amsmath,verbatim}
\DeclareMathOperator{\Ff}{Fact}
\DeclareMathOperator{\Pre}{Pref}

\DeclareMathOperator{\card}{card}
\DeclareMathOperator{\alf}{alph}

\newtheorem{Question}{Question}
 \title{A Coloring Problem for Sturmian and Episturmian Words}
 
\author{Aldo de Luca\inst{1}\and Elena V. Pribavkina\inst{2}\and Luca Q. Zamboni\inst{3}}
\institute{%
Dipartimento di Matematica\\
Universit\`a di Napoli Federico II, Italy\\
\email{aldo.deluca@unina.it}
\and
Ural Federal University, Ekaterinburg, Russia\\
\email{elena.pribavkina@usu.ru}
\and
Universit\'e Claude Bernard Lyon 1, France\\
and University of Turku, Finland\\
\email{lupastis@gmail.com}
}
\begin{document}
\maketitle
\begin{abstract} We consider the following open question in the spirit of Ramsey theory: Given an aperiodic infinite word $w$, does there exist a finite coloring of its factors such that no factorization of $w$ is monochromatic? We show that such a coloring always exists whenever $w$ is a Sturmian word or a standard episturmian word.
\end{abstract}

\section{Introduction}
 Ramsey theory (including Van der Waerden's theorem) (see \cite{GRS}) is a topic  of  great interest in combinatorics with connections to various fields of mathematics. A remarkable consequence of Ramsey's Infinitary Theorem applied to combinatorics on words yields    the following unavoidable regularity of infinite words\footnote{Actually, the proof of Theorem \ref{thm:schutz} given by Sch\"utzenberger in \cite{schutz} does not use Ramsey's theorem}: 

\begin{theorem}\label{thm:schutz} Let $A$ be a non-empty alphabet, $w$ be an infinite word over $A$, $C$ a finite non-empty set (the set of colors), and $c:  \Ff ^+ w \rightarrow C$ any coloring of the set  $\Ff ^+ w$ of all non-empty factors of $w$. Then there exists a factorization of $w$ of the form $w= VU_1U_2\cdots U_n\cdots$ such that for all positive integers $i$ and $j$, $c(U_i)=c(U_j)$.
\end{theorem}

One can ask whether given an infinite word there exists a suitable coloring map able to avoid  the  monochromaticity of {\em all}  factors in all factorizations of the word. 
More precisely,  the following variant of Theorem~\ref{thm:schutz} was posed as a question by T.C. Brown \cite{BTC} and, independently, by the third author \cite{LQZ}:

\begin{Question}\label{que:prima}. Let $w$ be an aperiodic  infinite   word over  a finite alphabet $A$. Does there exist a finite coloring
$ c:  \Ff ^+ w  \rightarrow C $
with the property that for any factoring $w= U_1U_2\cdots U_n \cdots$, there exist positive integers
$i,j$ for which $c(U_i)\neq c(U_j)$ ?
\end{Question}

Let us observe that for periodic words the answer to the preceding question is trivially  negative. Indeed, let $w=U^{\omega}$, and  $c:  \Ff ^+ w \rightarrow C$ be  any finite coloring. By factoring $w$ as $w= U_1U_2\cdots U_n\cdots$, where for all $i\geq 1$, $U_i=U$
one has $c(U_i)=c(U_j)$ for all positive integers $i$ and $j$. It is easy to see that there exist non recurrent infinite words $w$ and finite colorings such that for any factoring
$w=U_1U_2\cdots U_n\cdots$ there exist  $i\neq j$ for which $c(U_i)\neq c(U_j)$. For instance, consider the infinite word $w=ab^{\omega}$ and define the coloring map as follows: for any non-empty factor $U$ of $w$, $c(U)= 1$ if it contains $a$ and $c(U)=0$, otherwise. Then for any factoring
$w=U_1U_2\cdots U_n\cdots$, $c(U_1)=1$ and $c(U_i)= 0$ for all $i>1$.

It is not very difficult  to prove that there exist infinite recurrent words for which  Question \ref{que:prima} has a positive answer, for instance square-free, overlap-free words, and standard Sturmian words \cite{LQZ}. 

In this paper we show that Question \ref{que:prima} has a positive
answer  for every Sturmian word where the number of colors is equal to  $3$. This solves a problem raised in \cite{BTC} and in \cite{LQZ}.  The proof requires some noteworthy new combinatorial properties of Sturmian words. Moreover, we prove that the same result holds true for  aperiodic standard episturmian words by using a number of colors equal to the number of distinct letters occurring in the word plus one. 

For all definitions and notation not explicitly given in the paper, the reader is referred to the books of M. Lothaire \cite{LO,LO2}; for Sturmian words see \cite[Chap.~2]{LO2} and for episturmian words see \cite{DJP,JP} and the survey of J. Berstel \cite{BC}.

\section{Sturmian words}

There exist several equivalent definitions of  Sturmian words. In particular, we recall (see, for instance, Theorem  2.1.5 of \cite{LO2}) that an infinite word $s\in \{a,b\}^{\omega}$ is Sturmian if and only if it is aperiodic and {\em balanced}, i.e., for all factors $u$ and $v$ of $s$ such that $|u|=|v|$ one has:
$$ | |u|_x-|v|_x| \leq 1, \ x\in \{a,b\},$$
 where $|u|_x$ denotes the number of occurrences of the letter $x$ in $u$. Since a Sturmian word $s$ is aperiodic, it must have at least one of the two factors $aa$ and $bb$. However, from the balance property,  it follows that a Sturmian word cannot have  both the factors  $aa$ and $bb$.

\begin{definition} We say that a Sturmian word is of type $a$ (resp. $b$) if it does not contain the factor $bb$ (resp. $aa$).
\end{definition}
We recall that a factor $u$ of a finite or infinite word $w$ over the alphabet $A$ is called {\em right special} (resp. {\em left special}) if there exist two different letters $x,y\in A$ such that $ux, uy$ (resp. $xu, yu$) are factors of $w$.

 A different equivalent definition of a Sturmian word is the following: A binary infinite word $s$ is Sturmian if   for every  integer $n\geq 0$, $s$ has a unique left (or equivalently right) special factor  of length $n$. It follows from this that  $s$ is {\em closed under reversal}, i.e., if $u$ is a factor of $s$ so  is its reversal $u^{\sim}$.
 
 A Sturmian word $s$ is called {\em standard} (or {\em characteristic})  if  all its prefixes are left special factors of $s$.
  As is well known, for any Sturmian word $s$ there exists a standard Sturmian word $t$ such that
 $\Ff s = \Ff t$, where for any finite or infinite word $w$, $\Ff w$ denotes the set of all its factors including the empty word.

\begin{definition} Let $s\in \{a,b\}^{\omega}$ be a Sturmian word. A non-empty factor $w$ of $s$ is rich in the letter $z\in \{a, b\}$ if there exists a factor $v$ of $s$ such that $|v|=|w|$ and $|w|_z>|v|_z$.
\end{definition}
From the aperiodicity and the balance property of a Sturmian word one easily derives that any non-empty factor $w$ of a Sturmian word $s$ is  rich either in the letter $a$ or in the letter $b$ but not in both  letters.
Thus one can introduce for any given Sturmian word $s$
a map $$r_s :  \Ff ^+ s  \rightarrow \{a, b\}$$ defined as: for any non-empty factor $w$ of $s$,
$r_s(w)= z\in \{a, b\}$ if  $w$ is rich in the letter $z$. Clearly, $r_s(w)=r_s(w^{\sim})$ for any $w\in \Ff ^+ s$.

For any letter $z\in \{a, b\}$ we shall denote by
${\bar z}$ the complementary letter of $z$, i.e., ${\bar a}= b$ and ${\bar b}= a$.
\begin{lemma}\label{lemma:lastletter} Let $w$ be a non-empty right special (resp. left special) factor
of a Sturmian word $s$. Then $r_s(w)$ is equal to the first letter of $w$ (resp. $r_s(w)$ is equal to the last letter of $w$). 
\end{lemma}
\begin{proof} Write $w=zw'$ with $z\in \{a, b\}$ and $w'\in \{a, b\}^*$. Since $w$ is a right special factor of $s$ one has that $v=w'{\bar z}$ is a factor of $s$. Thus  $|w|= |v|$ and $|w|_z >|v|_z$, whence
$r_s(w)= z $. Similarly, if $w$ is left special one deduces that $r_s(w)$ is equal to the last letter of $w$.
\qed\end{proof}

\section{Preliminary Lemmas}

\begin{lemma}\label{theorem:primo} Let $s$ be a Sturmian word such that
$$ s= \prod_{i\geq 1} U_i,$$
where  the $U_i$'s are non-empty factors of $s$. If for every  $i$ and $j$, $r_s(U_i)= r_s(U_j)$, then
for any $M>0$ there exists an integer $i$ such that $|U_i|>M$.
\end{lemma}
\begin{proof} Suppose to the contrary that  for some positive integer $M$ we have that  $|U_i|\leq M$ for each $i\geq 1$. This
implies that the number of distinct $U_i$'s in the sequence $(U_i)_{i\geq 1}$ is finite, say $t$.
Let $ r_s(U_i)=x\in \{a, b\}$ for all $i\geq1$ and set for each $i\geq 1$:
$$ f_i = \frac{|U_i|_x}{|U_i|}.$$
Thus $\{f_i \mid i\geq 1\}$ is a finite set of at most $t$ rational numbers. We set $r= \min\{f_i \mid i\geq 1\}$.

Let $f_x(s)$ be the frequency of the letter $x$ in $s$ defined as 
$$ f_x(s)=\lim_{n\rightarrow\infty} \frac{|s_{[n]}|_x}{n},$$
where for every $n\geq1$, $s_{[n]}$ denotes the prefix of $s$ of length $n$. As is well known (see Prop. 2.1.11 of \cite{LO2}), $f_x(s)$ exists and is an irrational number.

Let us now prove that  $r >f_x(s)$. From Proposition 2.1.10 in \cite{LO2} one derives that for
all $V\in \Ff s$
$$  |V|f_x(s) -1 < |V|_x <   |V|f_x(s)+1.$$
Now for any $i\geq 1$, $U_i$ is rich in the letter $x$, so that there exists $V_i \in \Ff s$ such that
$|U_i|= |V_i|$ and $|U_i|_x>|V_i|_x$.
From the preceding inequality one has:
$$|U_i|_x = |V_i|_x+1>  |V_i|f_x(s)=  |U_i|f_x(s),$$
so that for all $i\geq 1$, $f_i> f_x(s)$, hence $r >f_x(s)$.

For any $n>0$, we can write the prefix $s_{[n]}$ of length $n$ as:
$$ s_{[n]}= U_1\cdots U_kU'_{k+1},$$
for a suitable $k\geq 0$ and $U'_{k+1}$ a prefix of $U_{k+1}$.
Thus
$$ |s_{[n]}|_x = \sum_{i=i}^k |U_i|_x+ |U'_{k+1}|_x.$$
Since $|U_i|_x= f_i|U_i|\geq r|U_i|$ and $|U'_{k+1}|\leq M$, one has
$$ |s_{[n]}|_x \geq r \sum_{i=1}^k |U_i| = r(n - |U'_{k+1}|) \geq rn -rM.$$
Thus
$$ \frac{|s_{[n]}|_x} {n} \geq r -r\frac{M}{n},$$
and 
$$ f_x(s)=\lim_{n\rightarrow\infty} \frac{|s_{[n]}|_x}{n} \geq r,$$
a contradiction.
\qed\end{proof}

In the following we shall consider the Sturmian morphism $R_a$, that we simply denote $R$,
defined as: 
\begin{equation}\label{morf:erre}
R(a)= a \  \mbox{and} \  R(b)=ba. 
\end{equation}

For any finite or infinite word $w$, $\Pre w$ will denote the set of all its prefixes. The following holds:

\begin{lemma}\label{lemma:onetwo} 
 Let $s$ be a Sturmian word and  $t\in \{a,b\}^{\omega}$ such that  $R(t)=s$. If either 
 \begin{itemize}

\item [1)] the first letter of $t$ (or, equivalently, of $s$) is $b$

\vspace{2mm}
 
  or
  
  \vspace{2 mm}

\item[ 2)] the Sturmian word $s$ admits a factorization:
$$s = U_1\cdots U_n\cdots,$$
where each $U_i$, $i\geq 1$, is a non-empty prefix of $s$ terminating in the letter $a$ and  $r_s(U_i)=r_s(U_j)$ for all $i,j \geq 1$, 
\end{itemize}
then $t$ is also Sturmian.
\end{lemma} 

\begin{proof} Let us prove that in both  cases $t$ is balanced. Suppose to the contrary that $t$ is unbalanced. 
Then (see Prop. 2.1.3 of \cite{LO2}) there would exists $v$ such that 
$$ava, bvb \in \Ff t.$$
Thus
$$aR(v)a, \ baR(v)ba \in \Ff s.$$
If $ava \not\in \Pre t$, then  $t= \lambda ava \mu$, with $\lambda\in \{a, b\}^+$ and $\mu\in  \{a,b\}^{\omega}$. Therefore $R(t)= R(\lambda)R(ava)R(\mu)$. Since the last letter of $R(\lambda)$ is $a$, it follows that  $aaR(v)a \in \Ff s$. As $baR(v)b\in \Ff s$ we reach a contradiction with the balance
property of $s$. In  case 1), $t$ begins in the letter $b$, so that $ava \not\in \Pre t$ and then $t$ is balanced.
In  case 2) suppose that $ava\in \Pre t$. This implies that $aR(v)a \in \Pre s$. From the preceding lemma in the factorization of  $s$ in prefixes there exists an integer $i>1$ such that $|U_i|> |aR(v)a|$.
Since $U_{i-1}$ terminates in $a$ and $U_{i-1}U_i \in\Ff s $, it follows that $aaR(v)a \in \Ff s$ and
one contradicts again the balance property of $s$. Hence, $t$ is balanced. 

Trivially, in both cases $t$ is aperiodic, so that $t$ is Sturmian.
\qed\end{proof}
Let us remark that, in general, without any additional  hypothesis, if $s=R(t)$, then $t$ need not  be Sturmian. For instance, if $f$ is the Fibonacci word $f= abaababaaba\cdots$, then $af$ is also a Sturmian word. However, it is readily verified that in this case
the word $t$ such that $R(t)=s$ is not balanced, so that $t$ is not Sturmian.

For any  finite or infinite word $w$ over the alphabet $A$, $\alf w$ denotes the set of all distinct letters of $A$ occurring in $w$.
We will make use of the following lemma:
\begin{lemma}\label{lemma:cp} Let $s$ be a Sturmian word  having a factorization
$$ s = U_1\cdots U_n \cdots,$$
where for $ i\geq 1$, $U_i$ are non-empty prefixes of $s$. Then for any $p\geq 1$, $U_1\neq c^p$
where $c$ is the first letter of $s$. 
\end{lemma}
\begin{proof} Suppose that $U_1=c^p$. Since $s$ is aperiodic there exists a minimal integer $j$
such that  $\card(\alf U_j) = 2$.  Since $U_j$ is a prefix of $s$, one has then  $U_1\cdots U_{j-1}U_j = U_j \xi$, with $\xi \in \{a, b\}^*$.  As $U_1\cdots U_{j-1} = c^q$ for a suitable $q\geq p$, it follows that
$\xi= c^q$ and $U_j\in cc^*$, a contradiction.
\qed\end{proof}

\section{Main results}

\begin{proposition}\label{prop:twocases} Let $s$ be a Sturmian word of type $a$ having a factorization
$$ s = U_1\cdots U_n \cdots,$$
where for $ i\geq 1$, $U_i$ are non-empty prefixes of $s$ such that $r_s(U_i)=r_s(U_j)$ for all $i,j\geq 1$.
Then one of the  following two properties holds:
\begin{itemize}
\item[i)] All $U_i$, $i \geq 1$, terminate in the letter $a$.
\item[ii)] For all $i\geq 1$, $U_ia \in \Pre s$.
\end{itemize}
\end{proposition}
\begin{proof} Let us first suppose that $s$ begins in the letter $b$. All prefixes $U_i$, $i\geq 1$, of $s$ begin in the letter $b$ and, as $s$ is of type $a$, have to terminate in the letter $a$. Thus in this case Property $i)$ is satisfied.

Let us then suppose that $s$ begins in the letter $a$. Now either all prefixes $U_i$, $i\geq 1$, terminate in the letter $a$ or  all prefixes $U_i$, $i\geq 1$, terminate in the letter $b$ or some of the prefixes terminate in the letter $a$ and some in the letter $b$.
We have then to consider the following cases:

\vspace{2 mm}

\noindent
Case 1. All prefixes $U_i$, $i\geq 1$, terminate in the letter $b$. 

Since $s$ is of type $a$, no one of the prefixes $U_i$, $i\geq 1$,  can be a right special factor. This implies that $U_ia \in \Pre s$ and
Property $ii)$ is satisfied.

\vspace{2 mm}

\noindent
Case 2. Some of the prefixes $U_i$, $i\geq 1$, terminate in the letter $a$ and some in the letter $b$.

We have to consider two subcases:

\vspace{2 mm}

\noindent
a).  $r_s(U_i)=b$, for all $i\geq 1$.

As all $U_i$, $i\geq 1$, begin in $a$, if any $U_i$ were right special,  then by Lemma \ref{lemma:lastletter},
$r_s(U_i)=a$, a contradiction. 
It follows that for all $i\geq 1$, $U_ia \in \Pre s$.

\vspace{2 mm}

\noindent
b).  $r_s(U_i)=a$, for all $i\geq 1$.

Some of the prefixes $U_j$, $j\geq 1$, terminate in $a$ (since otherwise we are in  Case 1).
Let  $U_k$ be a prefix terminating in $a$ for a suitable $k\geq 1$. If a prefix $U_i$ terminates
in $b$, then $aU_i$ is not a factor of $s$. Indeed, otherwise,  the word $a U_ib^{-1}$ is such that $|a U_ib^{-1}|= |U_i|$ and
$|a U_ib^{-1}|_b< |U_i|_b$, so that  $r_s(U_i)= b$ a contradiction. Thus one derives that all $U_l$
with $l\geq k$ terminate in $a$. Moreover, if some $U_i$ terminate in $b$, by Lemma  \ref{theorem:primo}
there exists $j >k$ such that $U_j$ has the prefix $U_i$, so that $U_{j-1}U_i \in \Ff s$. Since $U_{j-1}$
terminates in $a$, one has that
$aU_i$ is a factor of $s$, a contradiction. Thus all $U_i$, $i\geq 1$, terminate in $a$.
\qed\end{proof}

\begin{proposition}\label{prop:basic} Let $s$ be a Sturmian word  having a factorization
$$ s = U_1\cdots U_n \cdots,$$
where for $ i\geq 1$, $U_i$ are non-empty prefixes of $s$ such that $r_s(U_i)=r_s(U_j)$ for all $i,j\geq 1$. Then there exists a Sturmian word $t$ such that
$$t = V_1\cdots V_n \cdots,$$
where for all $i\geq 1$, $V_i$ are non-empty prefixes of $t$, $r_t(V_i)=r_t(V_j)$ for all $i,j\geq 1$, and
$|V_1|<|U_1|$.
\end{proposition}
\begin{proof} We can suppose without loss of generality that $s$ is a Sturmian word of type $a$.
From Proposition \ref{prop:twocases} either all  $U_i$, $i\geq 1$, terminate in the letter $a$
or for all $i\geq 1$, $U_ia \in \Pre s$. We  consider two cases:

\vspace{2 mm}

\noindent
Case 1.  For all $i\geq 1$, $U_ia \in \Pre s$. 

We can suppose that $s$ begins in the letter $a$. Indeed, otherwise, if the first letter of $s$ is $b$, then all $U_i$, $i\geq 1$, begin in the letter $b$ and, as $s$ is of type $a$, they have to terminate in the letter $a$. Thus the case that the first letter of $s$ is $b$ will be considered when we will analyze  Case 2.

We consider the injective endomorphism of $\{a, b\}^*$, $L_a$, or simply $L$,  defined by 
$$L(a)=a \ \mbox{and} \ \  L(b)=ab. $$ Since $s$ is of type $a$, the first letter of $s$ is $a$,  and $X=\{a, ab\}$ is a code having a finite deciphering delay (cf. \cite{BP}), the word $s$ can be uniquely factorized by the elements of $X$. Thus there exists a unique word $t \in \{a,b\}^{\omega}$ such that  $s = L(t)$. The following holds:
\begin{itemize}
\item[1.] The word $t$ is a Sturmian word.
\item[2.] For any $i\geq 1$ there exists a non-empty prefix $V_i$ of $t$ such that $L(V_i)= U_i$.
\item[3.] The word $t$ can be factorized as $ t= V_1\cdots V_n \cdots.$
\item[4.] $|V_1|<|U_1|$.
\item[5.] For all $i,j \geq 1$, $r_t(V_i)= r_t(V_j)$.
\end{itemize}
Point 1. This is a consequence of the fact that $L$ is a standard Sturmian morphism (see Corollary 2.3.3 in Chap. 2 of \cite{LO2}).

\vspace{2 mm}

\noindent
Point 2.  For any $i\geq 1$, since $U_ia\in \Pre s$  and any pair $(c, a)$ with $c\in \{a, b\}$ is synchronizing for $X^{\infty}=X^*\cup X^{\omega}$ (cf. \cite{BP}), one has  that $U_i\in X^*$, so that there exists $V_i\in \Pre t$ such that
$L(V_i)= U_i$.

\vspace{2 mm}

\noindent
Point 3. One has $L(V_1\cdots V_n\cdots) = U_1\cdots U_n \cdots= s = L(t)$. Thus
$t= V_1\cdots V_n \cdots$.

\vspace{2 mm}

\noindent
Point 4. By Lemma \ref{lemma:cp}, $U_1$ is not a power of $a$ so that in $U_1$ there must be at least one occurrence of the letter $b$. This implies that $|V_1|<|U_1|$.

\vspace{2 mm}

\noindent
Point 5. We shall prove that for all $i\geq 1$, $r_t(V_i)= r_s(U_i)$. From this  one has that for all $i,j \geq 1$, $r_t(V_i)= r_t(V_j)$.

Since $t$ is a Sturmian word, there exists $V'_i\in \Ff t$ such that 
$$ |V_i|= |V'_i| \ \mbox{and} \ \mbox{either} \ |V_i|_a> |V'_i|_a \ \mbox{or} \  |V_i|_a< |V'_i|_a .$$
In the first case $r_t(V_i)= a$ and in the second case $r_t(V_i)= b$. Let us set
$$ F_i = L(V'_i).$$
Since $U_i = L(V_i)$, from the definition of the morphism $L$ one has:
\begin{equation}\label{eq:uno}
|F_i|_a = |V'_i|_a+|V'_i|_b = |V'_i|, \ |F_i|_b= |V'_i|_b.
\end{equation}
\begin{equation}\label{eq:due}
|U_i|_a = |V_i|_a+|V_i|_b = |V_i|, \ |U_i|_b= |V_i|_b.
\end{equation}

Let us first consider the case $r_t(V_i)= a$, i.e., $|V_i|_a= |V'_i|_a+1$ and $|V_i|_b= |V'_i|_b-1$.
From the preceding equations one has:
$$ |F_i|= |U_i|+1.$$
Moreover, from the definition of $L$ one has that  $F_i$ begins in the letter $a$. Hence,
$|a^{-1}F_i|= |U_i|$ and $|a^{-1}F_i|_a= |F_i|_a-1= |U_i|_a-1$. Thus $|U_i|_a >|a^{-1}F_i|_a$.
Since $a^{-1}F_i \in \Ff s$, one has 
$$r_s(U_i)= r_t(V_i)= a .$$
Let us now consider the case $r_t(V_i)= b$, i.e., $|V_i|_a= |V'_i|_a-1$ and $|V_i|_b= |V'_i|_b+1$.
From (\ref{eq:uno}) and (\ref{eq:due}) one derives:
$$|U_i|= |F_i|+1,$$
and $|U_i|_b>|F_i|_b$. Now $F_ia$ is a factor of $s$. Indeed, $F_i= L(V'_i)$ and for any $c\in \{a, b\}$
such that $V'_ic \in \Ff t$ one has $L(V'_ic)= F_iL(c)$. Since for any letter $c$, $L(c)$ begins in the letter $a$ it follows that $F_ia \in \Ff s$. Since $|F_ia| = |U_i|$ and $|U_i|_b >|F_i|_b = |F_ia|_b$, one has that $U_i$ is rich in $b$. Hence, $r_s(U_i)= r_t(V_i)=b$.

\vspace{2 mm}

\noindent
Case 2. All  $U_i$, $i\geq 1$, terminate in the letter $a$.

We consider the injective endomorphism of $\{a, b\}^*$, $R_a$, or simply $R$,  defined in (\ref{morf:erre}). 
 Since  $s$ is of type $a$ and $X=\{a, ba\}$ is a  prefix code, the word $s$ can be uniquely factorized by the elements of $X$. Thus there exists a unique word $t \in \{a,b\}^{\omega}$ such that  $s = R(t)$. The following holds:

\begin{itemize}
\item[1.] The word $t$ is a Sturmian word.
\item[2.] For any $i\geq 1$ there exists a non-empty prefix $V_i$ of $t$ such that $R(V_i)= U_i$.
\item[3.] The word $t$ can be factorized as $ t= V_1\cdots V_n \cdots.$
\item[4.] $|V_1|<|U_1|$.
\item[5.] For all $i,j \geq 1$, $r_t(V_i)= r_t(V_j)$.
\end{itemize}

\noindent
Point 1. From Lemma \ref{lemma:onetwo}, since $R(t)=s$ it follows that $t$ is Sturmian.

\vspace{2 mm}

\noindent
Point 2.  For any $i\geq 1$, since $U_i$ terminates in the letter $a$  and any pair $(a, c)$ with $c\in \{a, b\}$ is synchronizing for $X^{\infty}$, one has  that $U_i\in X^*$, so that there exists $V_i\in \Pre t$ such that
$R(V_i)= U_i$.

\vspace{2 mm}

\noindent
Point 3. One has $R(V_1\cdots V_n\cdots) = U_1\cdots U_n \cdots= s = R(t)$. Thus
$t= V_1\cdots V_n \cdots$.

\vspace{2 mm} 

\noindent
Point 4. By Lemma \ref{lemma:cp}, $U_1$ is not a power of the first letter  $c$ of $s$, so that in $U_1$ there must be at least one occurrence of the letter $\bar c$. This implies that $|V_1|<|U_1|$.

\vspace{2 mm}

\noindent
Point 5. We shall prove that for all $i\geq 1$, $r_t(V_i)= r_s(U_i)$. From this  one has that for all $i,j \geq 1$, $r_t(V_i)= r_t(V_j)$.

Since $t$ is a Sturmian word, there exists $V'_i\in \Ff t$ such that 
$$ |V_i|= |V'_i| \ \mbox{and} \ \mbox{either} \ |V_i|_a> |V'_i|_a \ \mbox{or} \  |V_i|_a< |V'_i|_a .$$
In the first case $r_t(V_i)= a$ and in the second case $r_t(V_i)= b$. Let us set
$$ F_i = R(V'_i).$$
Since $U_i = R(V_i)$, from the definition of the morphism $R$ one has that equations (\ref{eq:uno}) and
(\ref{eq:due}) are satisfied. 

Let us first consider the case $r_t(V_i)= a$, i.e., $|V_i|_a= |V'_i|_a+1$ and $|V_i|_b= |V'_i|_b-1$.
From the preceding equations one has:
$$ |F_i|= |U_i|+1.$$
From the definition of the morphism $R$ one has that $F_i= R(V'_i)$ terminates in the letter $a$.
Hence, $|F_ia^{-1}|= |U_i|$ and $|F_ia^{-1}|_a=  |F_i|_a-1= |U_i|_a-1$. Thus $|U_i|_a=|F_ia^{-1}|_a+1$, so that $U_i$ is rich in $a$ and $r_s(U_i)= r_t(V_i)=a$.

Let us now suppose that  $r_t(V_i)= b$, i.e., $|V_i|_a= |V'_i|_a-1$ and $|V_i|_b= |V'_i|_b+1$. 
From (\ref{eq:uno}) and (\ref{eq:due}) one derives:
$$|U_i|= |F_i|+1,$$
and $|U_i|_b>|F_i|_b$.
We prove that $aF_i \in \Ff s$. Indeed,  $F_i= R(V'_i)$ and for any $c\in \{a, b\}$
such that $cV'_i \in \Ff t$ one has $R(c)R(V'_i)= R(c)F_i$. Note that such a letter $c$ exists always as $t$ is recurrent. Since for any letter $c$, $R(c)$ terminates in the letter $a$ it follows that $aF_i \in \Ff s$. Since $|aF_i| = |U_i|$ and $|U_i|_b >|aF_i|_b = |F_i|_b$, one has that $U_i$ is rich in $b$. Hence, $r_s(U_i)= r_t(V_i)=b$.
\qed\end{proof}

\begin{theorem}\label{theorem:basic1} Let $s$ be a Sturmian word  having a factorization
$$ s = U_1\cdots U_n \cdots,$$
where each $U_i$, $i\geq 1$, is a non-empty prefix of $s$. Then there exist integers $i,j \geq 1$ such that
$r_s(U_i)\neq r_s(U_j)$.
\end{theorem}
\begin{proof} Let $s$ be a Sturmian word and suppose that $s$ admits a factorization
$$ s = U_1\cdots U_n \cdots,$$
where for $ i\geq 1$, $U_i$ are non-empty prefixes such that for all $i,j \geq 1, $ $r_s(U_i)=r_s(U_j)$.
Among all Sturmian words having this property we can always consider a Sturmian word $s$ such that
$|U_1|$ is minimal.  Without loss of generality we can suppose that $s$ is of type $a$. By Proposition \ref{prop:basic}  there exists a Sturmian word $t$ such that
$$t = V_1\cdots V_n \cdots,$$
where for all $i\geq 1$, $V_i$ are non-empty prefixes, $r_t(V_i)=r_t(V_j)$ for all $i,j\geq 1$, and
$|V_1|<|U_1|$, that contradicts the minimality of the length of $U_1$.
\qed\end{proof}
\begin{theorem} Let $s$ be a Sturmian word. There exists a coloring $c$ of the non-empty factors of $s$,
$c: \Ff ^+ s \rightarrow \{0,1,2\}$ such that for any factorization
$$s = V_1\cdots V_n \cdots $$
in non-empty factors $V_i$, $i\geq 1$, there exist integers $i,j$ such that $c(V_i)\neq c(V_j)$.
\end{theorem}
\begin{proof} Let us define the coloring $c$ as: for any $V\in  \Ff^+ s$
$$ c(V)= \left \{\begin{array}{ll}
                 0   & \mbox{if $V$ is not a prefix of $s$}\\
                 1   & \mbox{if  $V$ is a prefix of $s$ and $r_s(V)= a$}\\
                 2   & \mbox{if  $V$ is a prefix of $s$ and $r_s(V)= b$}
                 \end{array}
                 \right . $$
Let us suppose to contrary that for all $i,j$, $c(V_i)=c(V_j)=x\in \{0,1,2\}$.  If $x=0$ we reach a contradiction as $V_1$ is a prefix of $s$ so that $c(V_1)\in \{1, 2\}$. If $x=1$ or $x=2$, then
all $V_i$ have to be prefixes of $s$ having the same richness, but this contradicts Theorem \ref{theorem:basic1}.
\qed\end{proof}
\section{The case of standard episturmian words}
An infinite word $s$ over the alphabet $A$ is called {\em standard episturmian} if it is closed under reversal and every left special factor of $s$ is a prefix of $s$. A word $s\in A^{\omega}$ is called {\em episturmian}  if there exists a standard episturmian $t\in A^{\omega}$ such that $\Ff s = \Ff t$.  We  recall the following facts about  episturmian words \cite{DJP,JP}:

\vspace{2 mm}

Fact 1. Every prefix of an aperiodic standard episturmian word $s$ is a left special factor of $s$. In particular an aperiodic standard episturmian  word on a two-letter alphabet is a standard Sturmian word.

\vspace{2 mm} 

Fact 2. If $s$ is a standard episturmian word with first letter $a$, then $a$ is  {\em separating}, i.e., for any $x,y\in A$ if $xy\in \Ff s$, then $a\in \{x, y\}$. 

\vspace{2 mm}

For each $x\in A$, let  $L_x$ denote the standard episturmian morphism \cite{JP}  defined for  any $y\in A$ by $L_x(y) = x$ if $y=x$ and $L_x(y)= xy$ for $x\neq y$.

\vspace{2 mm}

Fact 3. The infinite word $s\in A^{\omega}$ is standard episturmian if and only if there exist a standard
episturmian word $t$ and $a\in A$ such that  $s= L_a(t)$. Moreover, $t$ is unique and the first letter of $s$ is $a$.

\vspace{2 mm}

The following was proved in \cite{GR}:

\vspace{2 mm}

Fact 4.  A recurrent word $w$ over the alphabet  $A$ is episturmian if and only if for each factor $u$ of $w$, a letter $b$ exists (depending on $u$) such that $AuA\cap \Ff w \subseteq buA\cup Aub$.

\begin{definition} We say that a standard episturmian word $s$ is of type $a$, $a\in A$, if the first letter of $s$ is $a$.
\end{definition}

\begin{theorem}\label{theorem:basic2} Let $s$ be an aperiodic standard episturmian word over the alphabet $A$ and let $s = U_1U_2 \cdots$ be any factoring of $s$ with each $U_i$, $i\geq 1$, a non-empty prefix of $s$. Then there exist indices $i\neq j$ for which $U_i$ and $U_j$ terminate in a different letter.
\end{theorem}

\begin{proof} Suppose to the contrary that there exists an aperiodic standard episturmian word  $s$
over the alphabet $A$ admitting a factorization  $s = U_1U_2 \cdots$ in which all $U_i$ are non-empty  prefixes of $s$ ending in the same letter. Amongst all aperiodic standard episturmian words  over the  alphabet $A$ having the preceding factorization, we may choose one such $s$  for which $|U_1|$ is minimal. Let $a\in A$ be the first letter of $s$, so that $s$ is of type $a$.

Let us now prove that for every $i\geq 0$, one has that $U_ia$ is a prefix of $s$. Let us first suppose
that for all $i\geq 1$,  $U_i$ ends in a letter $x\neq a$. Since $a$ is separating  ($s$ is of type $a$),  $x$ can be followed only by $a$, so that the prefix $U_i$ can be followed only by $a$. This implies that $U_ia$ is a prefix of $s$.

Let us then suppose that for all $i\geq 1$,  $U_i$ ends in $a$.
 Since $U_1$ is a prefix of $s$, and all $U_i$, $i\geq 1$, begin in $a$
one has that $U_1a$ is a prefix of $s$. Now let $i> 1$. Since $U_{i-1}$ ends in $a$  it follows that $aU_ia$ is a factor of $s$.

Let $U_ix$ be a prefix of $s$; we want to
show that $x=a.$ Since $U_ix$ is left special (as it is a prefix of $s$), there exists a
letter $y\neq a$ such that $yU_ix$ is a factor of $s.$
Now from this and by Fact 4,
 there exists a letter $b$ (depending only on $U_i$) such that 
 either $x=b$ or $y=b.$

So now, by Fact 4, since $aU_ia$ and $yU_ix$ are both factors of $s$,  we deduce $b=a$ and
either $x=a$ or $y=a.$ Since $y\neq a,$ it follows that $x=a.$
Therefore, we have proved that  for every $i\geq 1$,  $U_ia$ is a prefix of $s$.

Let us now observe that $U_1$ must contain the occurrence of a letter $x\neq a$. Indeed, otherwise,
suppose that $U_1 = a^k$ and consider the least $i>1 $ such that $x$ occurs in $U_i$. This  implies, by using an argument similar to that of the proof of Lemma \ref{lemma:cp}, that $U_i$ cannot be a prefix of $s$.

\vspace{2 mm}

\noindent
By Fact 3, one has that there exists a unique standard episturmian word $s'$ such that
$ s = L_a(s')$
and $\alf s'  \subseteq \alf s \subseteq A$. Moreover, since $s$ is
aperiodic, trivially one has that also $s'$ is aperiodic.

Let us observe that the set  $X= \{a\} \cup \{ax \mid x\in A\}$ is a code having deciphering
delay equal to $1$ and that any pair $(x,a)$ with $x\in A$ is synchronizing for $X^{\infty}$. This implies
that $s$ can be uniquely factored by the words of $X$.  Moreover, since $U_ia$ is a prefix of $s$, from the synchronization property of $X^{\infty}$,
it follows that for each  $i \geq 1$,
       $$ U_i = L_a(U'_i), $$
where $U'_i$ is a prefix of $s'$. 
  From the definition of $L_a$ and the preceding formula,
one has that the last letter of $U_i$ is equal to the last letter of $U'_i$.

Moreover, 
$$L_a(U'_1\cdots U'_n\cdots)= U_1\cdots U_n \cdots = s= L_a(s').$$
Thus
 $s'= U'_1\cdots U'_n\cdots$,
 where each  $U'_i$, $i\geq 1$, is a non-empty prefix of $s'$  and for all $i,j \geq 1$,  $U'_i$ and $U'_j$ terminate in the same letter.
Since in $U_1= L_a(U'_1)$ there is the occurrence of a letter different from $a$ one obtains that $|U'_1|<|U_1|$ which is a contradiction.
\qed\end{proof}

Let us observe that in the case of a standard  Sturmian word, Theorem \ref{theorem:basic2} is an immediate consequence of Theorem \ref{theorem:basic1} and Lemma \ref{lemma:lastletter}. 

\begin{theorem} Let $s$ be an aperiodic standard episturmian word and let $k=\card(\alf s)$. There exists a coloring $c$ of the non-empty factors of $s$,
$c: \Ff ^+ s  \rightarrow \{0,1,\ldots, k \}$ such that for any factorization
$$s = V_1\cdots V_n \cdots $$
in non-empty factors $V_i$, $i\geq 1$, there exist integers $i,j$ such that $c(V_i)\neq c(V_j)$.
\end{theorem}
\begin{proof} Let $\alf s =\{a_1,\ldots, a_k\}$. We define the coloring $c$ as: for any $V\in  \Ff^+ s$
$$ c(V)= \left \{\begin{array}{ll}
                 0   & \mbox{if  $V$ is not a prefix of $s$} \\
                 1   & \mbox{if  $V$ is a prefix of $s$ terminating in $a_1$}\\
                 \vdots   &  \vdots\\
                 k   & \mbox{if $V$ is  a prefix of $s$ terminating in $a_k$}
                 \end{array}
                 \right . $$
Let us suppose by contradiction that for all $i,j$, $c(V_i)=c(V_j)=x\in \{0,1,\dots, k\}$.  If $x=0$ we reach a contradiction as $V_1$ is a prefix of $s$ so that $c(V_1)\in \{1, \ldots, k\}$. If $x\in \{1,\ldots,k\}$, then
all $V_i$ have to be prefixes of $s$ terminating in the same letter, but this contradicts Theorem \ref{theorem:basic2}.
\qed\end{proof}

\vspace{2 mm}

\noindent
{\bf Acknowledgments}

\vspace{2 mm}

\noindent
The authors are indebted to Tom Brown for his suggestions and comments. The second author acknowledges support from the Presidential Program for young researchers, grant MK-266.2012.1. The third author is partially supported by a FiDiPro grant from the Academy of Finland.

\end{document}